\documentclass[11pt]{amsart}

\usepackage{epigamath}

\usepackage[english]{babel}

\numberwithin{equation}{section}

\usepackage[shortlabels]{enumitem}

\usepackage{amsmath,amssymb,amsthm}
\usepackage{thmtools}
\usepackage{mathtools}

\usepackage[all]{xy}

\usepackage{tikz-cd}
\usetikzlibrary{babel}

\usepackage{url}

\theoremstyle{plain}

\newtheorem{thm}{Theorem}[section]
\newtheorem{prop}[thm]{Proposition}
\newtheorem{lem}[thm]{Lemma}
\newtheorem{cor}[thm]{Corollary}
\newtheorem{conj}[thm]{Conjecture}
\newtheorem*{claim}{Claim}

\theoremstyle{definition}
\newtheorem{dfn}[thm]{Definition}

\theoremstyle{remark}
\newtheorem{rem}[thm]{Remark}

\newtheorem{notation}[thm]{Notation}

\newcommand{\N}{\mathbb{N}}
\newcommand{\Q}{\mathbb{Q}}
\newcommand{\R}{\mathbb{R}}
\newcommand{\C}{\mathbb{C}}

\DeclareMathOperator{\codim}{codim}
\DeclareMathOperator{\mult}{mult}
\DeclareMathOperator{\Exc}{Exc}

\DeclareMathOperator{\Supp}{Supp}

\DeclareMathOperator{\nklt}{nklt}
\DeclareMathOperator{\sB}{\mathbf{B}}

\setlist[enumerate,1]{label={\rm(\roman*)}, ref={\rm\roman*}} 

\newlist{a-enumerate}{enumerate}{2}
\setlist[a-enumerate,1]{label={\rm(\alph*)}, ref={\rm\alph*}}

\newlength{\savedparindent}%
\EpigaVolumeYear{9}{2025} \EpigaArticleNr{24} \ReceivedOn{July 8, 2024}
\InFinalFormOn{May 23, 2025}
\AcceptedOn{June 25, 2025}
	
\title{Nakayama--Zariski decomposition and\\ the termination of flips}
\titlemark{Nakayama--Zariski decomposition and the termination of flips}

\author{Vladimir Lazi\'c}
\address{Fachrichtung Mathematik, Campus, Geb\"aude E2.4, Universit\"at des Saarlandes, 66123 Saarbr\"ucken, Germany}
\email{lazic@math.uni-sb.de}
		
\author{Zhixin Xie}
\address{Institut \'Elie Cartan de Lorraine, Universit\'e de Lorraine, 54506 Nancy, France}
\email{zhixin.xie@univ-lorraine.fr}

\authormark{V.\ Lazi\'c and Z.\ Xie}

\AbstractInEnglish{We show that for pseudoeffective projective pairs the termination of one sequence of flips implies the termination of all flips, assuming a natural conjecture on the behaviour of the Nakayama--Zariski decomposition under the operations of a Minimal Model Program.}



\acknowledgement{Lazi\'c gratefully acknowledges support by the Deutsche Forschungsgemeinschaft (DFG, German Research Foundation) -- Project-ID 286237555 -- TRR 195.}

\begin{document}



\maketitle

\begin{prelims}

\DisplayAbstractInEnglish

\bigskip

\DisplayKeyWords

\medskip

\DisplayMSCclass

\end{prelims}


\newpage

\setcounter{tocdepth}{1}

\tableofcontents


\section{Introduction}

In this paper we study the relationship between the existence of minimal models and the termination of flips conjecture in the Minimal Model Program (MMP). We show that -- for pseudoeffective projective pairs over $\C$ -- the termination of one sequence of flips implies the termination of them all, assuming a natural conjecture on the behaviour of the Nakayama--Zariski decomposition under the operations of an MMP.

\subsection*{Previous work}
Termination of flips and the existence of minimal models are two of the most important open problems in higher-dimensional birational geometry. They are both known in dimension $3$ by \cite{Sho85,Kaw92,Sho96}, whereas in higher dimensions the situation is more complicated. In dimension $4$ the existence of minimal models is known by \cite{LT22}, building on \cite{KMM87,Sho09,Bir12b}. The termination of flips is known for canonical $4$-folds by \cite{KMM87,Fuj05} and for \emph{pseudoeffective} log canonical pairs, \textit{i.e.}~for projective log canonical pairs $(X,\Delta)$ such that $K_X+\Delta$ is pseudoeffective, by \cite{CT23,Mor25}. In dimension $5$ the existence of minimal models is known for pseudoeffective log canonical pairs $(X,\Delta)$ whose underlying variety is uniruled, or which satisfy $\kappa(X,K_X+\Delta)\geq0$; see \cite{LT22}.

Apart from these unconditional results, there are several conditional results in higher dimensions, which we discuss next in some detail, as this is relevant for the results in the present paper. It is convenient at this point to consider pseudoeffective and non-pseudoeffective pairs separately. Whereas it is known that for non-pseudoeffective log canonical pairs there exists at least one MMP which terminates by \cite{BCHM,HH20}, there does not currently exist any strategy to attack the termination of flips in general for this class, apart from the difficult strategy of \cite{Sho04}.

The situation is somewhat different for pseudoeffective pairs $(X,\Delta)$. It was realised in \cite{Bir12b} that the existence of minimal models is related to the existence of a decomposition of the divisor $K_X+\Delta$ into -- roughly -- a sum of a nef and an effective divisor, called \emph{weak Zariski decomposition}; see for instance \cite[Definition 2.10]{LT22}. Building on this approach, we now understand how the existence of minimal models and, to some extent, the termination of flips are related to \emph{generalised pairs}: very roughly, generalised pairs are couples $(X,\Delta+M)$, where $(X,\Delta)$ is a usual pair and $M$ is a divisor which behaves similarly to a nef divisor; see Definition~\ref{dfn:g-pairs}. There are now several inductive results connecting the existence of weak Zariski decompositions to the existence of minimal models of generalised pairs, see \cite{LT22,LT22a,HL22,TX24}, or to the termination of flips of generalised pairs, see \cite{CT23,Mor25}. Generalised pairs are indispensable in all recent progress on the existence of minimal models and the termination of flips, even for usual pairs.

In particular, one of the main results of \cite{CT23} states that for pseudoeffective generalised pairs the existence of weak Zariski decompositions implies the termination of flips, \emph{assuming the termination of flips in lower dimensions}. As mentioned above, the problem is that we currently do not have an inductive statement which works for non-pseudoeffective pairs; a similar issue in the context of the existence of minimal models motivated \cite{LT22}, as explained in the introduction to \textit{op.\ cit.} Overcoming this issue is the main motivation for the present paper.

There is, to our knowledge, only one general approach to the termination of flips conjecture in full generality: that of \cite{Sho04}, based on two difficult conjectures on the behaviour of discrepancies. The methods and the logic of the current paper are completely different.

\subsection*{Balanced MMP}
Nakayama introduced in \cite{Nak04} an important generalisation in higher dimensions of the classical Zariski decomposition on surfaces; this is a fundamental tool in a large portion of the recent progress in the MMP. To any pseudoeffective $\R$-divisor $D$ on a $\Q$-factorial projective variety, one can associate the \emph{Nakayama--Zariski decomposition} $P_\sigma(D)+N_\sigma(D)$, which has several excellent properties; we recall its precise definition and the properties we need in Section~\ref{sec:prelim}.

If a pseudoeffective generalised pair $(X,B+M)$ has a minimal model, then the Nakayama--Zariski decomposition of $K_X+B+M$ is automatically its weak Zariski decomposition -- this is in a sense \emph{the canonical} choice for a weak Zariski decomposition; see Lemma~\ref{lem:resNQC}. However, whereas weak Zariski decompositions have been exploited in the context of the existence of minimal models and the termination of flips before, the good properties of the Nakayama--Zariski decomposition have not.

Now, let $(X_1,B_1+M_1)$ be a pseudoeffective generalised pair which has a minimal model, and consider a $(K_{X_1}+{B_1}+{M_1})$-MMP 
\begin{center}
\begin{tikzcd}[column sep = 0.8em, row sep = 1.75em]
(X_1,B_1+M_1) \arrow[rr, dashed, "\pi_1"] && (X_2,B_2+M_2) \arrow[rr, dashed, "\pi_2"] && (X_3,B_3+M_3) \arrow[rr, dashed, "\pi_3"] && \cdots ,
\end{tikzcd}
\end{center} 
and for each $i$ set $P_i:=P_\sigma(K_{X_i}+B_i+M_i)$ and $N_i:=N_\sigma(K_{X_i}+B_i+M_i)$. We say that the sequence is \emph{balanced} if for each $i$ the log canonical threshold of $P_i+N_i$ with respect to $(X_i,B_i+M_i)$ is zero; see Remark~\ref{rem:easycomputation} and Lemma~\ref{lem:lct}. 

Our first result is that (assuming the existence of minimal models) the termination of flips conjecture for pseudoeffective generalised pairs follows from the termination of flips conjecture for \emph{balanced} pseudoeffective generalised pairs. This gives proper context for our main result, Theorem~\ref{mainthm} below. For the acronym NQC, see Section~\ref{sec:prelim}. 

\begin{prop}\label{pro:balanced}
Let $ (X_1, B_1+M_1) $ be a projective NQC log canonical generalised pair of dimension $n$ such that $K_{X_1}+B_1+M_1$ is pseudoeffective, and assume that $ (X_1, B_1+M_1) $ has a minimal model. Consider a $(K_{X_1}+{B_1}+{M_1})$-MMP
\begin{center}
\begin{tikzcd}[column sep = 0.8em, row sep = 1.75em]
(X_1,B_1+M_1) \arrow[rr, dashed, "\pi_1"] && (X_2,B_2+M_2) \arrow[rr, dashed, "\pi_2"] && (X_3,B_3+M_3) \arrow[rr, dashed, "\pi_3"] && \cdots .
\end{tikzcd}
\end{center}
Then there exist a projective $\Q$-factorial NQC dlt generalised pair $ (X_1', B_1'+M_1') $ of dimension $n$ which has a minimal model and a balanced $(K_{X_1'}+B_1'+M_1')$-MMP
\begin{center}
\begin{tikzcd}[column sep = 0.8em, row sep = 1.75em]
\left(X_1',B_1'+M_1'\right) \arrow[rr, dashed, "\pi_1'"] && \left(X_2',B_2'+M_2'\right) \arrow[rr, dashed, "\pi_2'"] && \left(X_3',B_3'+M_3'\right) \arrow[rr, dashed, "\pi_3'"] && \cdots 
\end{tikzcd}
\end{center}
such that the $(K_{X_1}+{B_1}+{M_1})$-MMP terminates if and only if the $(K_{X_1'}+B_1'+M_1')$-MMP terminates.
\end{prop}

\subsection*{The main result}
We now propose the following conjecture.

\begin{conj}\label{con:mainconjecture}
Let $ (X_1, B_1+M_1) $ be a projective $\Q$-factorial NQC dlt generalised pair such that $K_{X_1}+B_1+M_1$ is pseudoeffective. Consider a $(K_{X_1}+B_1+M_1)$-MMP 
\begin{center}
\begin{tikzcd}[column sep = 0.8em, row sep = 1.75em]
(X_1,B_1+M_1) \arrow[rr, dashed, "\pi_1"] && (X_2,B_2+M_2) \arrow[rr, dashed, "\pi_2"] && (X_3,B_3+M_3) \arrow[rr, dashed, "\pi_3"] && \cdots. 
\end{tikzcd}
\end{center}
Let $S$ be a log canonical centre of\, $(X_1,B_1+M_1)$ which belongs to $\sB_-(K_{X_1}+B_1+M_1)$. Then there exists an index~$i$ such that $\pi_i$ is not an isomorphism at the generic point of the strict transform of $S$ on $X_i$.
\end{conj}

Here, $\sB_-$ denotes the \emph{diminished base locus}; we recall its definition and properties in Section~\ref{subsec:B-}. Conjecture~\ref{con:mainconjecture} is a consequence of the termination of flips conjecture; see Corollary~\ref{cor:B-termination}. Moreover, Conjecture~\ref{con:mainconjecture} is in fact equivalent to another statement, Conjecture~\ref{con:mainconjecture2}, which only involves components of the negative part of the Nakayama--Zariski decomposition of a dlt blowup of the generalised pair $ (X_1, B_1+M_1) $; this is proved in Proposition~\ref{pro:equivalence} below. Thus, though this is not obvious, Conjecture~\ref{con:mainconjecture} is actually a statement on the behaviour of the Nakayama--Zariski decomposition. This is an important point we want to convey in this paper.

The main result of this paper -- the following Theorem~\ref{mainthm} -- shows, together with Proposition~\ref{pro:balanced},  that Conjecture~\ref{con:mainconjecture} implies the termination of flips conjecture for pseudoeffective generalised pairs, assuming the existence of minimal models.

\begin{thm} \label{mainthm}
Assume the termination of flips for pseudoeffective NQC $\Q$-factorial dlt generalised pairs in dimension at most $n-1$.

Let $ (X_1, B_1+M_1) $ be a projective $\Q$-factorial NQC dlt generalised pair of dimension $n$ such that $K_{X_1}+B_1+M_1$ is pseudoeffective. Consider a balanced sequence of flips in a $(K_{X_1}+{B_1}+{M_1})$-MMP 
\begin{center}
\begin{tikzcd}[column sep = 0.8em, row sep = 1.75em]
(X_1,B_1+M_1) \arrow[rr, dashed, "\pi_1"] && (X_2,B_2+M_2) \arrow[rr, dashed, "\pi_2"] && (X_3,B_3+M_3) \arrow[rr, dashed, "\pi_3"] && \cdots .
\end{tikzcd}
\end{center}
Assume that Conjecture~\ref{con:mainconjecture} holds for this sequence of flips. If the generalised pair $ (X_1, B_1+M_1) $ has a minimal model, then the sequence terminates.
\end{thm}

In order to describe briefly the strategy of the proof of Theorem~\ref{mainthm} and the role of the Nakayama--Zariski decomposition in it, we recall that the proofs of \cite[Theorem 1.3]{Bir07} and of \cite[Theorem 1.3]{CT23} proceed by contradiction and are based on the following two steps: (i) the so-called \emph{special termination} (this is the statement that the non-klt locus of the generalised pair eventually becomes disjoint from the flipping loci), and (ii) the completion of the proof by using special termination to construct an infinite strictly increasing sequence of log canonical thresholds which violates the ACC property of log canonical thresholds. Even if one is interested only in the termination of flips for pseudoeffective generalised pairs, non-pseudoeffective generalised pairs appear in both of these steps: in the first step because restrictions to log canonical centres can produce non-pseudoeffective generalised pairs in \emph{lower} dimensio
ns, and in the second step because modifying boundaries of generalised pairs may produce non-pseudoeffective generalised pairs in the \emph{same} dimension.

Considering Nakayama--Zariski decompositions instead of weak Zariski decompositions helps us resolve the first issue above by achieving, through a careful argument, that relevant log canonical centres do not come from the negative part of the Nakayama--Zariski decomposition. This is done in Section~\ref{sec:specterm}. The issue in the second step seems to be more subtle, and a simple modification of the arguments of \cite{CT23} does not seem to do the trick. Instead, we rely crucially on properties of the Nakayama--Zariski decomposition to show that the sequence in the theorem cannot be balanced, thus deriving a contradiction. This is done in Section~\ref{sec:mainresult}. We stress that Conjecture~\ref{con:mainconjecture} is fundamental in both of these steps.

\subsection*{Acknowledgments}
We would like to thank Nikolaos Tsakanikas for many useful comments, suggestions, and discussions, and J\'anos Koll\'ar, who pointed out that results in a previous version of this paper had to be stated more precisely. We would also like to thank the referees for extensive comments and suggestions.

\section{Preliminaries}\label{sec:prelim}

Throughout the paper, unless otherwise stated, we work with varieties that are normal, quasi-projective, and defined over $\C$.

If $f\colon X\to Z $ is a projective morphism between normal varieties and if $ D_1 $ and $ D_2 $ are two $ \R $-Cartier $\R$-divisors on $ X $, then $ D_1 $ and $ D_2 $ are \emph{$\R$-linearly equivalent over $ Z $}, denoted by $ D_1 \sim_{\R,Z} D_2 $, if there exists an $\R$-Cartier $\R$-divisor $G$ on $Z$ such that $D_1\sim_\R D_2+f^*G$.
An $\R$-divisor $D$ on $X$ is an \emph{NQC divisor} over $Z$ if it is a non-negative linear combination of $\Q$-Cartier divisors on $X$ which are nef over $Z$. The acronym NQC stands for \emph{nef\, $\Q$-Cartier combinations}.

We refer to \cite{KM98} for the definitions and basic results on the singularities of usual pairs and the MMP. We often quote the negativity lemma \cite[Lemma 3.39]{KM98}.

\subsection{An auxiliary result}
The following is a small generalisation of \cite[Lemma II.5.6]{Nak04} that we will use in the proof of Theorem~\ref{thm:specterm_g-pairs1}. We provide the proof for the benefit of the reader.

\begin{lem}\label{lem:pullbackpseff}
    Let $f\colon X\to Y$ be a surjective morphism of projective varieties, and let $D$ be an $\R$-Cartier $\R$-divisor on $X$. Then $f^* D$ is pseudoeffective if and only if $D$ is pseudoeffective.
\end{lem}

\begin{proof}
    One direction is immediate. For the converse, assume that $f^* D$ is pseudoeffective. Let $\theta\colon Y'\to Y$ be a desingularisation, and let $X'$ be a desingularisation of the main component of the fibre product $X\times_Y Y'$ with  induced morphisms $f'\colon X'\to Y'$ and $\theta'\colon X'\to X$.
\[
\xymatrix{ 
X' \ar[d]_{\theta'} \ar[r]^{f'} & Y' \ar[d]^{\theta}\\
X\ar[r]_f & Y.
}
\]
Then $\theta'^*f^*D$ is pseudoeffective, and since $\theta'^*f^*D=f'^*\theta^*D$, the divisor $\theta^*D$ is pseudoeffective by \cite[Lemma II.5.6]{Nak04}. Therefore, by replacing $f$ by $\theta$, it suffices to show the lemma when $f$ is birational.

To this end, let $A$ be an ample divisor on $Y$. Since for every $\varepsilon>0$ we have $\kappa(X,f^*D+\varepsilon f^*A)=\kappa(Y,D+\varepsilon A)$ by \cite[Lemma II.3.11]{Nak04}, it follows that all divisors $D+\varepsilon A$ are big. Therefore, $D$ is pseudoeffective.
\end{proof}

\subsection{Generalised pairs}

Generalised pairs, which are the main objects of the present paper, were introduced in \cite{BH14,BZ16}. Here we recall their definition and basic properties we need; we refer to \cite{HL22,LT22a} for a more detailed treatment.

\begin{dfn}\label{dfn:g-pairs}
	A \emph{generalised pair} or \emph{g-pair} $(X/Z,B+M)$ consists of a variety $ X $ equipped with  projective morphisms 
	$$  X' \overset{f}{\longrightarrow} X \longrightarrow Z , $$ 
	where $ f $ is birational and $ X' $ is normal, an effective $ \R $-divisor $B$ on $X$, and an $\R$-Cartier $\R$-divisor $M'$ on~$X'$ which is nef over $Z $ such that $ f_* M' = M $ and $ K_X + B + M $ is $ \R $-Cartier. We denote the g-pair simply by $ (X/Z,B+M) $ but implicitly remember the whole \emph{g-pair data} $X'\to X\to Z$ and $M'$. If additionally $M'$ is an NQC divisor on $ X' $, then the g-pair $(X/Z,B+M)$ is an \emph{NQC g-pair}.	
\end{dfn}

Note that in the above definition the variety $X'$ may always be chosen as a sufficiently high birational model of $X$. Unless otherwise stated, the variety $Z$ will always be assumed to be a point in this paper, in which case we denote the g-pair by $(X,B+M)$.

We now recall various classes of singularities of g-pairs.

\begin{dfn}\label{dfn:singularities}
	Let $ (X/Z,B+M) $ be a g-pair with data $ X' \overset{f}{\to} X \to Z $ and $ M' $. We may write 
	$$ K_{X'} + B' + M' \sim_\R f^* ( K_X + B + M ) $$
	for some $ \R $-divisor $ B' $ on $ X' $. Let $ E $ be a divisorial valuation over $ X $. We may assume that $E$ is a prime divisor on $ X' $; we denote its centre on $X$ by $c_X(E)$. The \emph{discrepancy of\, $ E $} with respect to $ (X,B+M) $ is $$ a (E, X, B+M) := {-} \mult_E B'.$$
	We say that the g-pair $ (X,B+M) $ is:   
	\begin{a-enumerate}
		\item \emph{klt} if $a (E, X, B+M) > -1 $ for all divisorial valuations $E$ over $X$; 
		\item \emph{log canonical} if $a (E, X, B+M) \geq -1 $ for all divisorial valuations $E$ over $X$; 
		\item \emph{dlt} if it is log canonical and if there exists an open subset $U\subseteq X$ such that the pair $(U,B|_U)$ is log smooth and such that if $a(E,X,B+M) = {-}1$ for some divisorial valuation $E$ over $X$, then the set $c_X(E)\cap U$ is non-empty and is a log canonical centre of $(U,B|_U)$.
	\end{a-enumerate}
	
		If $ (X,B+M) $ is a log canonical g-pair, then 
	\begin{enumerate}
		\item an irreducible subvariety $ S $ of $ X $ is a \emph{log canonical centre} of $ X $ if there exists a divisorial valuation $ E $ over $ X $ such that $$ a(E,X,B+M) = -1  \quad\text{and}\quad c_X(E) = S;$$
		
		\item the \emph{non-klt locus} of $ (X,B+M) $, denoted by $ \nklt(X,B+M) $, is the union of all log canonical centres of $ X $.
	\end{enumerate}
\end{dfn}

If $(X,B+M)$ is a log canonical g-pair, then the number of its log canonical centres is finite. It follows from the definition that if $ (X,B+M) $ is a dlt g-pair with $ \lfloor B \rfloor = 0 $, then $ (X,B+M) $ is klt. If $ (X,B+M) $ is a $ \Q $-factorial dlt g-pair, then, by definition and by \cite[Remark 4.2.3]{BZ16}, the underlying pair $ (X,B) $ is dlt and the log canonical centres of $ (X,B+M) $ coincide with those of $ (X,B) $. In particular, by \cite[Proposition~3.9.2]{Fuj07a} we obtain
	\[ \nklt(X,B+M) = \nklt(X,B) = \Supp \lfloor B \rfloor, \]
	and all log canonical centres of $(X,B+M)$ are normal. 

\begin{rem}\label{rem:easycomputation}
Let $ (X,B+M) $ be a g-pair with data $ X' \overset{f}{\to} X \to Z $ and $ M' $, let $P'$ be an $\R$-divisor on $X'$ which is nef over $Z$, set $P:=f_*P'$, and let $N$ be an effective $ \R $-divisor on $ X $. Set $F := f^*P - P'$. Then by the negativity lemma the $\R$-divisor $F$ is effective and $f$-exceptional.
An easy calculation shows that, for each prime divisor $E$ on $X'$ and for every real number $t$, we have
$$a\left(E, X, (B+tN)+(M+tP)\right) = a(E, X, B+M) - t\mult_E(f^*N + F).$$
Assume that 
$$\lambda:=\sup\left\{t\in\R\mid\left( X, (B+tN)+(M+tP)\right)\text{ is log canonical}\right\}$$
is a real number. Then $\lambda$ is the \emph{log canonical threshold of $P+N$ with respect to $(X,B+M)$}.
\end{rem}

The following easy lemma shows that the log canonical threshold is almost always well defined; see also  \cite[paragraph after Definition 4.3]{BZ16}.

\begin{lem}\label{lem:lct}
Let $ (X/Z,B+M) $ be a projective $\Q$-factorial g-pair, let $P$ be an $\R$-divisor on $X$ which is the pushforward of an $\R$-divisor on a high birational model of\, $X$ that is nef over $Z$, and let $N$ be an effective $ \R $-divisor on $ X $. Then there exists a positive real number $t$ such that $(X,(B+tN)+(M+tP))$ is not log canonical, unless $P$ is nef over $Z$ and $N=0$.
\end{lem}

\begin{proof}
Let $f\colon X'\to X$ be a log resolution of $(X,B+N)$ such that there exist $\R$-divisors $M'$ and $P'$ on $X'$ which are nef over $Z$, with $M = f_*M'$ and $P=f_*P'$. Set $F := f^*P - P'$. Then by Remark~\ref{rem:easycomputation} the g-pair $(X,(B+tN)+(M+tP))$ will not be log canonical for $t\gg0$ whenever $f^*N\neq0$ or $F\neq0$, which implies the lemma.
\end{proof}

\begin{notation}
We will use the following notation throughout the paper. Let $(X,B+M)$ be a dlt g-pair, and let~$S$ be a log canonical centre of $(X,B+M)$. We define a dlt g-pair $(S,B_S +M_S)$ by adjunction, \textit{i.e.}~by the formula
$$K_S + B_S +M_S = (K_X + B +M)|_S$$
as in \cite[Definition 2.8, Proposition 2.10]{HL22}.
\end{notation}

\subsection{Dlt blowups and MMP}

The Minimal Model Program has recently been generalised to the category of g-pairs in \cite{HL23,LT22a}. The foundational results are analogous to those for the usual pairs; we refer to \cite[Sections 2.2 and 2.4]{LT22a} for a brief summary. We also refer to \cite[Section 2]{LMT23} for the definitions of minimal and log canonical models that we use in this paper.

\begin{notation}
Given a g-pair $(X_1,B_1+M_1)$ and a $(K_{X_1}+B_1+M_1)$-MMP
\begin{center}
		\begin{tikzcd}[column sep = 0.8em, row sep = 1.75em]
			(X_1,B_1+M_1) \arrow[rr, dashed, "\pi_1"] && (X_2,B_2+M_2) \arrow[rr, dashed, "\pi_2"] && (X_3,B_3+M_3) \arrow[rr, dashed, "\pi_3"] && \cdots,
		\end{tikzcd}
	\end{center}
the notation means implicitly that each $B_i$ is the strict transform of $B_1$, and that for each $i$ the divisors $M_1$ and $M_i$ are pushforwards of the same nef $\R$-divisor on a common birational model of $X_1$ and $X_i$.
\end{notation}

We  often use the next result on the existence of \emph{dlt blowups}; \textit{cf.} \cite[Proposition 3.10]{HL22} and \cite[Lemma~4.5]{BZ16}. We recall its proof as we will need the construction in some proofs later in the paper.

\begin{lem}\label{lem:dltblowup}
	Let $(X,B+M)$ be a log canonical g-pair with data $ \widehat X \overset{f}{\to} X \to Z $ and $ \widehat M $. Then, after possibly replacing $\widehat{X}$ with a higher model, there exist a $\Q$-factorial dlt g-pair $(X',B'+M')$ with data $ \widehat X \overset{g}{\to} X' \to Z $ and~$ \widehat M $, and a projective birational morphism $\pi \colon X' \to X$ such that 
	$$K_{X'} + B' + M' \sim_\R \pi^*(K_X + B +M) \quad\text{and}\quad B' = \pi_*^{-1}B + E,$$
 where $E$ is the sum of all the $\pi$-exceptional prime divisors. The g-pair $(X',B'+M')$ is a \emph{dlt blowup} of\, $(X,B+M)$.
	
	Furthermore, let $E_1,\dots,E_r$ be divisorial valuations over $X$ such that $a(E_i,X,B+M)={-}1$ for each $i$. Then we may assume that each $E_i$ is a divisor on $X'$ which is a component of\, $E$. 
\end{lem}

\begin{proof}
By passing to a higher birational model, we may assume that $f$ is a log resolution of $(X,B+M)$, \textit{i.e.}~$\widehat{X}$ is smooth and $\Supp(B)\cap \Supp(M)$ is a simple normal crossing divisor, such that each $E_i$ is a divisor on~$\widehat X$, and we write
    \begin{equation}\label{eq:logres}
        K_{\widehat X} + \widehat B + \widehat M  \sim_\R f^* (K_X + B + M )
    \end{equation}
    for an $ \R $-divisor $ \widehat B $ on $ \widehat X $. Set 
    $$F := \sum \left(1+a(P, X,B+M)\right) P,$$
    where the sum runs over all $f$-exceptional prime divisors $P$ on $\widehat X$. Then $F\geq0$, the support of $F$ does not contain any $E_i$, the g-pair $(\widehat X,(\widehat B+F)+\widehat M) $ is dlt, and by \eqref{eq:logres} we have
    \[
    K_{\widehat X} + \left(\widehat B+F\right) + \widehat M \sim_{\R,X} F.
    \]
    By \cite[Lemma 3.4]{HL22} we can run a $(K_{\widehat X} + (\widehat B+F) + \widehat M)$-MMP over $X$ with scaling of a divisor which is ample over $X$. By \cite[Proposition 3.9 and Lemma 3.6]{HL22}, this MMP contracts $F$ and terminates with a $\Q$-factorial dlt g-pair $(X',B'+M')$. Note that the exceptional locus at every step of this MMP is contained in the support of the strict transform of $F$; hence $F$ is precisely the divisor contracted in the MMP. Let $\pi\colon X'\to X$ be the induced birational map. The construction implies that $B' = \pi_*^{-1}B + E$, where $E$ is the sum of all the $\pi$-exceptional prime divisors, and the strict transform of each $E_i$ is contained in $E$. Therefore,~$\pi$ is the desired map.
\end{proof}

In the following lemma it is necessary to consider ample small quasi-flips instead of usual flips, as will become apparent in Step~\ref{th51-step3} of the proof of Theorem~\ref{thm:specterm_g-pairs1}. Recall that an \emph{ample small quasi-flip} is a commutative diagram
		\begin{center}
			\begin{tikzcd}
				(X,B+M) \arrow[rr, "\varphi", dashed] \arrow[dr, "f" swap] && (X',B'+M')\rlap{,} \arrow[dl, "f'"] \\
				& Y 
			\end{tikzcd}
		\end{center}
	where $(X,B+M)$ and $(X',B'+M')$ are g-pairs and $Y$ is normal, $f$ and $f'$ are projective birational morphisms and $\varphi$ is an isomorphism in codimension~$1$, $ \varphi_* B = B' $, $ M $ and $ M' $ are pushforwards of the same nef $\R$-divisor on a common birational model of $X$ and $X'$, $ {-}(K_X + B + M) $ is $f$-ample, and $ K_{X'} + B' + M' $ is $f'$-ample.

Lemma~\ref{lem:lifting_g} is similar to \cite[Lemma 3.2]{LMT23} and \cite[Theorem 2.14]{LT22a}; it removes the assumption in lower dimensions from \cite[Lemma 3.2]{LMT23} and makes that result more precise. We will use the lemma repeatedly in the paper: it allows us to pass from a sequence of ample small quasi-flips of log canonical g-pairs to an MMP of $\Q$-factorial dlt g-pairs starting from any given dlt blowup.
	
\begin{lem}\label{lem:lifting_g}
Let $ (X_1,B_1+M_1) $ be an NQC log canonical g-pair, and let $ f_1\colon (X_1',B_1'+M_1') \to (X_1,B_1+M_1)$ be a dlt blowup of\,  $ (X_1,B_1+M_1) $. Consider a sequence of ample small quasi-flips 
	\begin{center}
		\begin{tikzcd}[column sep = 0.8em, row sep = 1.75em]
			(X_1,B_1+M_1) \arrow[dr, "\theta_1" swap] \arrow[rr, dashed, "\pi_1"] && (X_2,B_2+M_2) \arrow[dl, "\theta_1^+"] \arrow[dr, "\theta_2" swap] \arrow[rr, dashed, "\pi_2"] && (X_3,B_3+M_3) \arrow[dl, "\theta_2^+"] \arrow[rr, dashed, "\pi_3"] && \cdots\rlap{.} \\
			& Z_1 && Z_2
		\end{tikzcd}
	\end{center}
	Then there exists a diagram
	\begin{center}
		\begin{tikzcd}[column sep = 0.8em, row sep = large]
			\left(X_1',B_1'+M_1'\right) \arrow[d, "f_1" swap] \arrow[rr, dashed, "\rho_1"] && \left(X_2',B_2'+M_2'\right) \arrow[d, "f_2" swap] \arrow[rr, dashed, "\rho_2"] && \left(X_3',B_3'+M_3'\right) \arrow[d, "f_3" swap] \arrow[rr, dashed, "\rho_3"] && \cdots 
			\\ 
			(X_1,B_1+M_1) \arrow[dr, "\theta_1" swap] \arrow[rr, dashed, "\pi_1"] && (X_2,B_2+M_2) \arrow[dl, "\theta_1^+"] \arrow[dr, "\theta_2" swap] \arrow[rr, dashed, "\pi_2"] && (X_3,B_3+M_3) \arrow[dl, "\theta_2^+"] \arrow[rr, dashed, "\pi_3"] && \cdots \rlap{,}\\
			& Z_1 && Z_2
		\end{tikzcd}
	\end{center}
	where for each $i \geq 1$ the map $ \rho_i \colon X_i' \dashrightarrow X_{i+1}' $ is a $(K_{X_i'}+B_i'+M_i')$-MMP over $ Z_i $ and the map $f_i $ is a dlt blowup of the g-pair $ (X_i,B_i+M_i) $. 
	
	In particular, the sequence at the top of the above diagram is an MMP for an NQC $\Q$-factorial dlt g-pair $ (X_1',B_1'+M_1') $.
\end{lem}

\begin{rem}
Throughout the paper, we will refer to diagrams as in Lemma~\ref{lem:lifting_g} as \emph{lifted MMP diagrams}.
\end{rem}

\begin{proof}[Proof of Lemma~\ref{lem:lifting_g}]
    The proof follows by repeating verbatim the beginning of the proof of \cite[Lemma~2.13]{LT22a}. We reproduce it here for the benefit of the reader. 
    
		By the definition of ample small quasi-flips and by \cite[Lemma 2.8]{LMT23}, the g-pair $(X_2,B_2+M_2)$ is a minimal model of $(X_1,B_1+M_1)$ over $Z_1$; hence by \cite[Lemma 3.10]{HL23} the g-pair $ (X_1',B_1'+M_1') $ has a minimal model in the sense of Birkar--Shokurov over $ Z_1 $. Therefore, by \cite[Lemma 2.10]{LT22a} there exists a $ (K_{X_1'} + B_1'+M_1') $-MMP with scaling of an ample divisor over $Z_1$ which terminates with a minimal model $ (X_2',B_2'+M_2') $ of $ (X_1',B_1'+M_1') $ over $Z_1$. Since $(X_2,B_2+M_2)$ is the log canonical model of $(X_1',B_1'+M_1')$ over $ Z_1 $, by \cite[Lemma 2.12]{LMT23} there exists a morphism $ f_2 \colon X_2' \to X_2 $ such that 
		$$ K_{X_2'} + B_2'+M_2' \sim_\R f_2^*\left(K_{X_2} + B_2 + M_2\right) . $$
		In particular, $ (X_2',B_2'+M_2') $ is a dlt blowup of $ (X_2,B_2+M_2) $. By continuing this procedure analogously, we obtain the result.
\end{proof}

We will need the following lemma in the proofs in Section~\ref{sec:specterm}. It is extracted from the proof of \cite[Theorem 5.1]{LMT23}; we have added more details for the benefit of the reader.

\begin{lem}\label{lem:valuations}	
	Let $ (X_1,B_1+M_1) $ be a $ \Q $-factorial NQC dlt g-pair, and consider a sequence of flips:
	\begin{center}
		\begin{tikzcd}[column sep = 0.8em, row sep = 1.75em]
			(X_1,B_1+M_1) \arrow[dr, "\theta_1" swap] \arrow[rr, dashed, "\pi_1"] && (X_2,B_2+M_2) \arrow[dl, "\theta_1^+"] \arrow[dr, "\theta_2" swap] \arrow[rr, dashed, "\pi_2"] && (X_3,B_3+M_3) \arrow[dl, "\theta_2^+"] \arrow[rr, dashed, "\pi_3"] && \cdots. \\
			& Z_1 && Z_2
		\end{tikzcd}
	\end{center}
	Let $S_1$ be a log canonical centre of\, $(X_1,B_1+M_1) $. For each $i\geq1$ assume inductively that $\pi_i|_{S_i}$ is an isomorphism in codimension $1$, and let $S_{i+1}$ be the strict transform of\, $S_i$ on $X_{i+1}$. Then there exists a positive integer $N$ such that 
	$$(\pi_i|_{S_i})_*B_{S_i}=B_{S_{i+1}}\quad\text{for every } i\geq N.$$
\end{lem}

\begin{proof}
We first claim that
\begin{equation}\label{eq:30}
\Supp B_{S_{i+1}}\subseteq \Supp (\pi_i|_{S_i})_*B_{S_i}\quad\text{for every } i.
\end{equation}
Indeed, without loss of generality we may assume that $i=1$. Assume that $E_2$ is a component of $B_{S_2}$ such that its strict transform $E_1$ on $S_1$ is not a component of $B_{S_1}$. Then 
$$0=a\left(E_1,S_1,B_{S_1} +M_{S_1}\right) \leq a\left(E_2,S_2,B_{S_2} +M_{S_2}\right)={-}\mult_{E_2}B_{S_2}<0 $$
by \cite[Lemma 2.8(iv)]{LMT23}, giving a contradiction. This proves \eqref{eq:30}.

In particular, since $B_{S_1}$ has finitely many components, by \eqref{eq:30} and by relabelling, we may assume that
\begin{equation}\label{eq:30a}
\Supp B_{S_{i+1}}=\Supp \left(\pi_i|_{S_i}\right)_*B_{S_i}\quad\text{for every } i.
\end{equation}

Now, let $E_1$ be a component of $B_{S_1}$, and set inductively $E_{i+1}:=(\pi_i|_{S_i})_*E_i$ for all $i$. Then $E_i$ is a component of $B_{S_i}$ by \eqref{eq:30a}; hence 
	\begin{equation}\label{eq:3}
	a\left(E_i,S_i,B_{S_i}+M_{S_i}\right)={-}\mult_{E_i}B_{S_i}<0\quad\text{for all }i.
	\end{equation}
Therefore, to prove the lemma it suffices to show that
	\begin{equation}\label{eq:3a}
	a\left(E_i,S_i,B_{S_i} +M_{S_i}\right) = a\left(E_{i+1},S_{i+1},B_{S_{i+1}}+M_{S_{i+1}}\right) \quad\text{for all }i\gg0.
	\end{equation}
	Assume towards a contradiction that \eqref{eq:3a} does not hold. Since $a(E_i,S_i,B_{S_i} +M_{S_i}) \leq a(E_{i+1},S_{i+1},B_{S_{i+1}}+M_{S_{i+1}})$ for all $i$ by \cite[Lemma 2.8(iv)]{LMT23}, we infer by \cite[Lemma 2.16(i)]{LMT23} that there exist a $\gamma\in(0,1)$ and infinitely many distinct discrepancies $a(E_i,S_i,B_{S_i}+M_{S_i})\geq{-}\gamma$. Thus, by \eqref{eq:3} the sequence $\{\mult_{E_i}B_{S_i}\}$ contains infinitely many distinct values not larger than $\gamma$. But this contradicts \cite[Lemma 2.16(ii)]{LMT23}, as can be seen by taking into account \cite[Definition 2.15]{LMT23}.
\end{proof}

\subsection{Nakayama--Zariski decomposition}

Given a smooth projective variety $X$ and a pseudoeffective $\R$-divisor $D$ on $X$, Nakayama \cite{Nak04} defined a decomposition $ D = P_\sigma (D) + N_\sigma (D) $, known as \emph{the Nakayama--Zariski decomposition} of $D$. This decomposition can be extended to the singular setting; see for instance \cite[Section 4]{BH14} and \cite[Section 2.2]{Hu20}. Here we consider only the case of $\Q$-factorial varieties and refer to \cite{Nak04} for the proofs of basic properties of the decomposition.

\begin{dfn}
    Let $X$ be a $\Q$-factorial projective variety, and let $\Gamma$ be a prime divisor on $X$. If $D$ is a big $\R$-divisor on $X$, then we set $|D|_\R:=\{\Delta\geq 0 \mid \Delta \sim_\R D\}$ and 
    \[\sigma_\Gamma (D) := \inf \left\{ \mult_\Gamma \Delta \mid \Delta\in |D|_\R \right\}.\]
    If $D$ is a pseudoeffective $\R$-divisor on $X$, then we pick an ample $\R$-divisor $A$ on $X$ and define
    \[\sigma_\Gamma (D) := \lim_{\varepsilon\downarrow 0} \sigma_\Gamma (D+\varepsilon A).\]
   Note that $\sigma_\Gamma(D)$ does not depend on the choice of $A$ and is compatible  with the definition above for big divisors. Set
   \[N_\sigma (D) := \sum_\Gamma \sigma_\Gamma(D)\cdot \Gamma\quad\text{and}\quad P_\sigma:=D-N_\sigma(D),\] 
   where the formal sum runs through all prime divisors $\Gamma$ on $X$. Both $N_\sigma(D)$ and $P_\sigma(D)$ are $\R$-divisors on $X$, and the decomposition $ D = P_\sigma (D) + N_\sigma (D) $ is the \emph{Nakayama--Zariski decomposition} of $D$.
\end{dfn}

The next lemma shows that the above definition is equivalent to the more general definition in \cite[Section 4]{BH14} when the underlying variety is $\Q$-factorial.

\begin{lem}\label{lem:Psigma}
  Let $X$ be a $\Q$-factorial projective variety, and let $D$ be a pseudoeffective $\R$-divisor on $X$. Let $f\colon Y\to X$ be a desingularisation of $X$. Then
  \[P_\sigma (D)= f_*P_\sigma (f^*D) \quad\text{and}\quad N_\sigma(D)= f_* N_\sigma(f^*D).\] 
\end{lem}

\begin{proof}
It suffices to show the second equality. If $A$ is an ample $\R$-Cartier divisor on $X$, then there is a bijection between the sets $|D+ \varepsilon A|_\R$ and $|f^*D+ \varepsilon f^*A|_\R$ for any $\varepsilon>0$; see for instance the proof of \cite[Lemma 2.3]{LMT23}. Therefore, for any prime divisor $\Gamma$ on $X$ and its strict transform $\Gamma':= f_*^{-1}\Gamma$ on $Y$, we have
$$ \sigma_{\Gamma}(D+\varepsilon A) = \sigma_{\Gamma'}(f^*D+\varepsilon f^*A). $$
Letting $\varepsilon\to 0$ and using \cite[Lemma III.1.7(2)]{Nak04}, we obtain $ \sigma_{\Gamma}(D) = \sigma_{\Gamma'}(f^*D)$, which shows the lemma.
\end{proof}

One of the most important properties of the Nakayama--Zariski decomposition relevant for this paper is contained in the following lemma; a generalisation is in Lemma~\ref{lem:restriction2} below.

\begin{lem}\label{lem:restriction}
  Let $X$ be a $\Q$-factorial projective variety, and let $D$ be a pseudoeffective $\R$-divisor on $X$. Let $\Gamma$ be a prime divisor on $X$ which is not a component of $N_\sigma(D)$. Assume that $\Gamma$ is a normal variety. Then the divisor $D|_\Gamma$ is pseudoeffective.
\end{lem}

Here and elsewhere in the paper, we interpret pseudoeffectivity in the sense of the $\R$-linear equivalence class; see also \cite[Section II.2, p.\ 37 and Remark II.5.8]{Nak04}.

\begin{proof}
Let $f\colon Y\to X$ be a desingularisation of $X$ and let $\Gamma':= f_*^{-1}\Gamma$. Applying Lemma~\ref{lem:pullbackpseff} to the morphism $f|_{\Gamma'}\colon \Gamma'\to \Gamma$, we see that, by replacing $X$ by $Y$ and $D$ by $f^*D$, it suffices to assume that $X$ is smooth. 

Fix an ample divisor $A$ on $X$. By \cite[Lemma III.1.7(3)]{Nak04} for each $\varepsilon>0$ there exists an $\R$-divisor $\Delta_\varepsilon\in|D+\varepsilon A|_\R$ such that $\Gamma\nsubseteq\Supp\Delta_\varepsilon$. In particular, $(D+\varepsilon A)|_\Gamma\sim_\R\Delta_\varepsilon|_\Gamma$ is pseudoeffective, and we let~$\varepsilon\to 0$.
\end{proof}

We will need the following lemma throughout the paper. In particular, part~\eqref{lem:resNQC-a} shows that a pseudoeffective g-pair $(X,B+M)$ has a minimal model if and only if the Nakayama--Zariski decomposition of $K_X+B+M$ is its weak Zariski decomposition.

\begin{lem}\label{lem:resNQC}
   Let $(X,B+M)$ be a projective $\Q$-factorial NQC log canonical g-pair such that $K_X+B+M$ is pseudoeffective. If $\varphi\colon (X,B+M)\dashrightarrow (X',B'+M')$ is a composition of divisorial contractions and flips in a $(K_X+B+M)$-MMP, then:
   \begin{a-enumerate}
   \item\label{lem:resNQC-a} $(X,B+M)$ has a minimal model if and only if there exists a desingularisation $f\colon Y\to X$ such that $P_\sigma( f^*(K_X+B+M)) $ is an NQC divisor on $Y$;
   \item\label{lem:resNQC-b} $\varphi_*N_\sigma(K_X+B+M)=N_\sigma(K_{X'}+B'+M')$, and if\, $(X',B'+M')$ is a minimal model of\, $(X,B+M)$, then $\varphi_*N_\sigma(K_X+B+M)=0$;
    \item\label{lem:resNQC-c} if\, $P$ is a prime divisor contracted by $\varphi$, then $P\subseteq\Supp N_\sigma (K_X+B+M)$;
    \item\label{lem:resNQC-d} $(X,B+M)$ has a minimal model if and only if\, $(X',B'+M')$ has a minimal model; 
    \item\label{lem:resNQC-e} if\, $(Z,B_Z+M_Z)$ is a projective $\Q$-factorial NQC log canonical g-pair with a birational morphism $\pi\colon Z\to X$ and if there exists an effective $\pi$-exceptional $\R$-divisor $E$ on $Z$ with the property that 
$$K_Z+B_Z+M_Z\sim_\R \pi^*(K_X+B+M)+E,$$
then $(X,B+M)$ has a minimal model if and only if\, $(Z,B_Z+M_Z)$ has a minimal model.
   \end{a-enumerate} 
\end{lem}

\begin{proof}
For~\eqref{lem:resNQC-a}, see for instance \cite[Theorem I]{TX24},
whereas~\eqref{lem:resNQC-b} follows from \cite[Lemma 4.1(5)]{BH14}. For~\eqref{lem:resNQC-c}, let $(p,q)\colon W\to X\times X'$ be a smooth resolution of indeterminacies of $\varphi$. By the negativity lemma there exists an effective $q$-exceptional $\R$-divisor $E$ on $W$ such that
$$p^*(K_X+B+M)\sim_\R q^*(K_{X'}+B'+M')+E$$
and $p_*^{-1}P\subseteq\Supp E$. Then by Lemma~\ref{lem:Psigma} and by \cite[Lemma~2.4]{LP20a}, we have
$$N_\sigma(K_X+B+M)=p_*N_\sigma\left(p^*(K_X+B+M)\right)=p_*N_\sigma\left(q^*(K_{X'}+B'+M')\right)+p_*E,$$
which proves~\eqref{lem:resNQC-c}, and
$$P_\sigma\left(p^*(K_X+B+M)\right)\sim_\R P_\sigma\left(q^*(K_{X'}+B'+M')\right).$$
Therefore, we obtain~\eqref{lem:resNQC-d} by applying~\eqref{lem:resNQC-a} to $(X,B+M)$ and to $(X',B'+M')$. Finally, for~\eqref{lem:resNQC-e} we note that for each desingularisation $\rho\colon W\to Z$ we have $P_\sigma(\rho^*(K_Z+B_Z+M_Z))\sim_\R P_\sigma(\rho^*\pi^*(K_X+B+M))$ by \cite[Lemma~2.4]{LP20a}, and we conclude by applying~\eqref{lem:resNQC-a} to $(X,B+M)$ and to $(Z,B_Z+M_Z)$.
\end{proof}

We will need the following remark several times in this paper.

\begin{rem}\label{rem:easycomputation2}
Let $ (X,B+M) $ be a $\Q$-factorial g-pair such that there exists a log resolution $f\colon X'\to X$ of $(X,B+M)$, where the $\R$-divisor $P':=P_\sigma(f^*(K_X+B+M))$ is NQC. We have $P:=P_\sigma(K_X+B+M)=f_*P'$ by Lemma~\ref{lem:Psigma}, and by passing to a higher resolution and by \cite[Corollary III.5.17]{Nak04}, we may assume that there exists an NQC divisor $M'$ on $X'$ such that $M = f_*M'$. Setting $F := f^*P - P'$ and $N:=N_\sigma(K_X+B+M)$, we have
$$f^*N+F=f^*(N+P) - P'=N_\sigma\left(f^*(K_X+B+M)\right);$$
hence by Remark~\ref{rem:easycomputation}, for each prime divisor $E$ on $X'$ and for every real number $t$, we have
$$a\left(E, X, (B+tN)+(M+tP)\right) = a(E, X, B+M) - t\mult_E N_\sigma\left(f^*(K_X+B+M)\right).$$
\end{rem}

\subsection{Diminished base locus}\label{subsec:B-}

If $X$ is a projective variety and if $D$ is a pseudoeffective $\R$-Cartier $\R$-divisor on $X$, then the \emph{stable base locus} of~$D$ is
$$\sB(D):=\bigcap_{D'\in|D|_\R}\Supp D',$$
whereas the \emph{diminished base locus} of $D$ is
$$\sB_-(D):=\bigcup_{A\text{ ample on }X}\sB(D+A);$$
the diminished base locus only depends on the numerical equivalence class of $D$ and is a countable union of closed subsets of $X$. We will need a few known results on the relationship between the diminished base locus and the negative part of the Nakayama--Zariski decomposition.

\begin{lem}\label{lem:pullbackdiminished}
Let $f\colon Y\to X$ be a surjective morphism between normal projective varieties, and let $D$ be a pseudoeffective $\R$-Cartier $\R$-divisor on $X$. If\, $X_\mathrm{sing}$ is the singular locus of\, $X$, then 
$$f^{-1}\sB_-(D)\cup f^{-1}\left(X_\mathrm{sing}\right)=\sB_-(f^*D)\cup f^{-1}\left(X_\mathrm{sing}\right)$$
and $\sB_-(f^*D)\subseteq f^{-1}\sB_-(D)$.
\end{lem}

\begin{proof}
The first statement is \cite[Proposition 2.5]{Leh13}, whereas the second statement follows from the first paragraph of the proof of \textit{loc.\ cit.}
\end{proof}

\begin{lem}\label{lem:nakayamazariskidiminished}
Let $X$ be a $\Q$-factorial projective variety, and let $D$ be a pseudoeffective $\R$-divisor on $X$. Then $N_\sigma(D)$ is the divisorial part of\, $\sB_-(D)$.
\end{lem}

\begin{proof}
Let $f\colon Y\to X$ be a desingularisation of $X$, and let $X_\mathrm{sing}$ be the singular locus of $X$. Let $\Gamma$ be a prime divisor in $\sB_-(D)$, and set $\Gamma':=f_*^{-1}\Gamma$. Then $\Gamma'\nsubseteq f^{-1}(X_\mathrm{sing})$ as $\codim_X X_\mathrm{sing}\geq2$, and since
$$f^{-1}\sB_-(D)\cup f^{-1}(X_\mathrm{sing})=\sB_-(f^*D)\cup f^{-1}(X_\mathrm{sing})$$
by Lemma~\ref{lem:pullbackdiminished}, we have that $\Gamma'\subseteq\sB_-(f^*D)$. Then $\Gamma'\subseteq \Supp N_\sigma(f^*D)$ by \cite[Theorem V.1.3]{Nak04};  hence $\Gamma\subseteq\Supp N_\sigma(D)$ by Lemma~\ref{lem:Psigma}, as desired.
\end{proof}

The following generalisation of Lemma~\ref{lem:restriction} will be crucial in the proof of Theorem~\ref{thm:specterm_g-pairs1}.

\begin{lem}\label{lem:restriction2}
  Let $X$ be a $\Q$-factorial projective variety, and let $D$ be a pseudoeffective $\R$-divisor on $X$. Let $S$ be a normal subvariety of\, $X$ which is not a subset of\, $\sB_-(D)$. Then the divisor $D|_S$ is pseudoeffective.
\end{lem}

\begin{proof}
Let $f\colon Y\to X$ be a log resolution of $S$ in $X$ which factorises through the blowup of $X$ along $S$. Then there exists a smooth prime $f$-exceptional divisor $T$ on $X$ such that $f(T)=S$. Since 
$$\Supp N_{\sigma}(f^*D)\subseteq\sB_-(f^*D)\subseteq f^{-1}\sB_-(D)$$
by Lemmas~\ref{lem:pullbackdiminished} and~\ref{lem:nakayamazariskidiminished}, and since $S\nsubseteq\sB_-(D)$ by assumption, we conclude that 
$$ T\nsubseteq \Supp N_{\sigma}(f^*D). $$
Therefore, by Lemma~\ref{lem:restriction} we have that $(f^*D)|_T$ is pseudoeffective; hence $D|_S$ is pseudoeffective, as can be seen by applying Lemma~\ref{lem:pullbackpseff} to the surjective morphism $f|_T\colon T\to S$.
\end{proof}

\section{On Conjecture~\ref{con:mainconjecture}}

In this section we prove several technical results which explain how Conjecture~\ref{con:mainconjecture} applies after we change an MMP either by lifting it to another MMP or by relabelling the indices in an MMP. Along the way we show several results mentioned in the introduction.

We start with an easy lemma which shows that in several crucial proofs in this paper we are allowed to \emph{relabel the indices} in a sequence of maps in an MMP.

\begin{lem}\label{lem:shiftingtheconjecture}
     Let $ (X_1,B_1+M_1) $ be a projective $\Q$-factorial NQC dlt g-pair such that $K_{X_1}+B_1+M_1$ is pseudoeffective. Consider a $(K_{X_1}+{B_1}+{M_1})$-MMP 
\begin{center}
\begin{tikzcd}[column sep = 0.8em, row sep = 1.75em]
(X_1,B_1+M_1) \arrow[rr, dashed, "\pi_1"] && (X_2,B_2+M_2) \arrow[rr, dashed, "\pi_2"] && (X_3,B_3+M_3) \arrow[rr, dashed, "\pi_3"] && \cdots .
\end{tikzcd}
\end{center}
If Conjecture~\ref{con:mainconjecture} holds for this MMP, then for each index $i$ Conjecture~\ref{con:mainconjecture} holds for the part of this MMP starting with the g-pair $(X_i,B_i+M_i)$.
\end{lem}

\begin{proof}
Fix an index $i$, and let $S_i$ be a log canonical centre of $(X_i, B_i+M_i)$ which belongs to $\sB_-(K_{X_i}+B_i+M_i)$. Let $E$ be a divisorial valuation over $X_i$ such that $c_{X_i}(E)=S_i$ and $a(E,X_i,B_i+M_i)={-}1$. By \cite[Lemma~2.8(i)]{LMT23} we have $a(E,X_j,B_j+M_j)={-}1$ for all $1\leq j\leq i$; hence by descending induction on $j$ and by \cite[Lemma~2.8(iii)]{LMT23} we conclude that each map $\pi_j^{-1}$ is an isomorphism at the generic point of the strict transform of $S_i$ on $X_j$, for $1\leq j\leq i-1$. Thus, we may define the strict transform $S_1$ of $S_i$ on $X_1$. Since we assume Conjecture~\ref{con:mainconjecture} for the given MMP, there exists an index $k\geq i$ such that $\pi_k$ is not an isomorphism at the generic point of the strict transform of $S_1$ on~$X_k$; hence $\pi_k$ is not an isomorphism at the generic point of the strict transform of $S_i$ on $X_k$, as desired.
\end{proof}

The following lemma shows that, in several crucial proofs in this paper, we are allowed to pass to an MMP which \emph{lifts} another MMP as in a lifted MMP diagram.

\begin{lem}\label{lem:liftingtheconjecture}
     Let $ (X_1,B_1+M_1) $ be a projective $\Q$-factorial NQC dlt g-pair such that $K_{X_1}+B_1+M_1$ is pseudoeffective. Consider a lifted MMP diagram 
   	\begin{center}
		\begin{tikzcd}[column sep = 0.8em, row sep = large]
			\left(X_1',B_1'+M_1'\right) \arrow[d, "f_1" swap] \arrow[rr, dashed, "\rho_1"] && \left(X_2',B_2'+M_2'\right) \arrow[d, "f_2" swap] \arrow[rr, dashed, "\rho_2"] && \left(X_3',B_3'+M_3'\right) \arrow[d, "f_3" swap] \arrow[rr, dashed, "\rho_3"] && \cdots 
			\\ 
			(X_1,B_1+M_1) \arrow[rr, dashed, "\pi_1"] && (X_2,B_2+M_2) \arrow[rr, dashed, "\pi_2"] && (X_3,B_3+M_3) \arrow[rr, dashed, "\pi_3"] && \cdots\rlap{,}
		\end{tikzcd}
	\end{center}
	where the sequence at the bottom of the diagram is a sequence of flips. Then Conjecture~\ref{con:mainconjecture} holds for the MMP at the bottom of this diagram if and only if Conjecture~\ref{con:mainconjecture} holds for the sequence at the top of the diagram.
\end{lem}

\begin{proof}
\setlength{\parindent}{0pt}
\begin{enumerate}[wide, label={{\it Step}~\rm\arabic*.}, ref=\arabic*]
\item\label{step1}
  To show one direction, assume that Conjecture~\ref{con:mainconjecture} holds for the MMP at the bottom of the diagram above. Let $S'$ be a log canonical centre of $(X_1', B_1'+M_1')$ which belongs to $\sB_-(K_{X_1'}+B_1'+M_1')$, and set $S:=f_1(S')$. Since
$$\sB_-\left(K_{X_1'}+B_1'+M_1'\right)\subseteq f_1^{-1}\sB_-\left(K_{X_1}+B_1+M_1\right)$$
  by Lemma~\ref{lem:pullbackdiminished}, we conclude that $S$ is a log canonical centre of $(X_1, B_1+M_1)$ which belongs to the set   $\sB_-(K_{X_1}+B_1+M_1)$. Let $E$ be a divisorial valuation over $X_1'$ such that $c_{X_1'}(E)=S'$ and $a(E,X_1',B_1'+M_1')={-}1$; hence $c_{X_1}(E)=S$, and $a(E,X_1,B_1+M_1)={-}1$ as the map $f_1$ is a dlt blowup. Then since we assume Conjecture~\ref{con:mainconjecture} for the $ (K_{X_1} + B_1+M_1) $-MMP at the bottom of the diagram above, there exists an index $i$ such that $\pi_i$ is not an isomorphism at the generic point of the strict transform of $S$ on $X_i$. This implies $a(E,X_i,B_i+M_i)>{-}1$ by  \cite[Lemma 2.8(iii)]{LMT23}; hence $a(E,X_i',B_i'+M_i')>{-}1$ as the map $f_i$ is a dlt blowup. But then again by  \cite[Lemma 2.8(iii)]{LMT23} applied to the $ (K_{X_1'} + B_1'+M_1') $-MMP at the top of the diagram above, we conclude that there exists a step of that MMP which is not an isomorphism at the generic point of the strict transform of $S'$, as desired.
\end{enumerate}
\setlength{\parindent}{\savedparindent}
\begin{enumerate}[resume, wide, label={{\it Step}~\rm\arabic*.}, ref=\arabic*]
\item\label{step2}
To show the converse, assume that Conjecture~\ref{con:mainconjecture} holds for the MMP at the top of the diagram above. Let $S$ be a log canonical centre of $(X_1,B_1+M_1)$ which belongs to $\sB_-(K_{X_1}+B_1+M_1)$. Assume, towards a contradiction, that each map $\pi_i$ is an isomorphism at the generic point of the strict transform of~$S$ on $X_i$. Let $E$ be a divisorial valuation over $X_1$ such that $c_{X_1}(E)=S$ and $a(E,X_1,B_1+M_1)={-}1$; hence $a(E,X_1',B_1'+M_1')={-}1$ as the map $f_1$ is a dlt blowup. Set $S':=c_{X_1'}(E)$. Since the g-pair $(X_1,B_1+M_1)$ is dlt,  $S$ does not belong to the singular locus $Z_1$ of $X_1$ by Definition~\ref{dfn:singularities}; hence $S'\nsubseteq f_1^{-1}(Z_1)$. As $f_1$ is a dlt blowup, by Lemma~\ref{lem:pullbackdiminished} we have
   $$f^{-1}_1\sB_-\left(K_{X_1}+B_1+M_1\right)\cup f_1^{-1}(Z_1)=\sB_-\left(K_{X_1'}+B_1'+M_1'\right)\cup f_1^{-1}(Z_1),$$
hence $S'\subseteq \sB_-(K_{X_1'}+B_1'+M_1')$. Then, since we assume Conjecture~\ref{con:mainconjecture} for the $ (K_{X_1'} + B_1'+M_1') $-MMP at the top of the diagram above, there exists an index $i$ such that $\rho_i$ is not an isomorphism at the generic point of the strict transform of $S'$ on $X_i'$. This implies $a(E,X_i',B_i'+M_i')>{-}1$ by  \cite[Lemma 2.8(iii)]{LMT23}, hence $a(E,X_i,B_i+M_i)>{-}1$ as the map $f_i$ is a dlt blowup. But then again by  \cite[Lemma 2.8(iii)]{LMT23} applied to the $ (K_{X_1} + B_1+M_1) $-MMP at the bottom of the diagram above, we conclude that there exists a step of that MMP which is not an isomorphism at the generic point of the strict transform of $S$, giving a contradiction which finishes the proof.\hfill\qedhere
\end{enumerate}
\renewcommand{\qed}{}
\end{proof}

Now we come to the important point mentioned in the introduction: that even though this is far from obvious, the statement of Conjecture~\ref{con:mainconjecture} is actually about the behaviour of the Nakayama--Zariski decomposition. In fact, we show in Proposition~\ref{pro:equivalence} below that it is equivalent to the following conjecture when we consider -- as we may -- MMPs consisting only of flips.

\begin{conj}\label{con:mainconjecture2}
     Let $ (X_1, B_1+M_1) $ be a projective $\Q$-factorial NQC dlt g-pair such that $K_{X_1}+B_1+M_1$ is pseudoeffective. Consider a sequence of flips in a $(K_{X_1}+{B_1}+{M_1})$-MMP
\begin{center}
\begin{tikzcd}[column sep = 0.8em, row sep = 1.75em]
(X_1,B_1+M_1) \arrow[rr, dashed, "\pi_1"] && (X_2,B_2+M_2) \arrow[rr, dashed, "\pi_2"] && (X_3,B_3+M_3) \arrow[rr, dashed, "\pi_3"] && \cdots ,
\end{tikzcd}
\end{center}
and consider a lifted MMP diagram 
	\begin{center}
		\begin{tikzcd}[column sep = 0.8em, row sep = large]
			\left(X_1',B_1'+M_1'\right) \arrow[d, "f_1" swap] \arrow[rr, dashed, "\rho_1"] && \left(X_2',B_2'+M_2'\right) \arrow[d, "f_2" swap] \arrow[rr, dashed, "\rho_2"] && \left(X_3',B_3'+M_3'\right) \arrow[d, "f_3" swap] \arrow[rr, dashed, "\rho_3"] && \cdots 
			\\ 
			(X_1,B_1+M_1) \arrow[rr, dashed, "\pi_1"] && (X_2,B_2+M_2) \arrow[rr, dashed, "\pi_2"] && (X_3,B_3+M_3) \arrow[rr, dashed, "\pi_3"] && \cdots\rlap{.}
		\end{tikzcd}
	\end{center}
Let $T$ be a component of\, $N_\sigma(K_{X'_1}+B'_1+M'_1)$ which is a log canonical centre of\, $(X'_1,B'_1+M'_1)$. Then the MMP at the top of the above diagram contracts $T$.
\end{conj}

\begin{prop}\label{pro:equivalence}
    Let $ (X_1,B_1+M_1) $ be a projective $\Q$-factorial NQC dlt g-pair such that $K_{X_1}+B_1+M_1$ is pseudoeffective. Consider a sequence of flips in a $(K_{X_1}+{B_1}+{M_1})$-MMP 
\begin{center}
\begin{tikzcd}[column sep = 0.8em, row sep = 1.75em]
(X_1,B_1+M_1) \arrow[rr, dashed, "\pi_1"] && (X_2,B_2+M_2) \arrow[rr, dashed, "\pi_2"] && (X_3,B_3+M_3) \arrow[rr, dashed, "\pi_3"] && \cdots .
\end{tikzcd}
\end{center}
   Then Conjecture~\ref{con:mainconjecture} holds for this sequence if and only if Conjecture~\ref{con:mainconjecture2} holds for any lifted MMP diagram
   	\begin{center}
		\begin{tikzcd}[column sep = 0.8em, row sep = large]
			\left(X_1',B_1'+M_1'\right) \arrow[d, "f_1" swap] \arrow[rr, dashed, "\rho_1"] && \left(X_2',B_2'+M_2'\right) \arrow[d, "f_2" swap] \arrow[rr, dashed, "\rho_2"] && \left(X_3',B_3'+M_3'\right) \arrow[d, "f_3" swap] \arrow[rr, dashed, "\rho_3"] && \cdots 
			\\ 
			(X_1,B_1+M_1) \arrow[rr, dashed, "\pi_1"] && (X_2,B_2+M_2) \arrow[rr, dashed, "\pi_2"] && (X_3,B_3+M_3) \arrow[rr, dashed, "\pi_3"] && \cdots\rlap{.}
		\end{tikzcd}
	\end{center}
\end{prop}

\begin{proof}
   If Conjecture~\ref{con:mainconjecture} holds for the given MMP, then Conjecture~\ref{con:mainconjecture2} holds for any  diagram as above by Lemmas~\ref{lem:liftingtheconjecture} and~\ref{lem:nakayamazariskidiminished}.

   To show the converse, let $S$ be a log canonical centre of $(X_1,B_1+M_1)$ which belongs to $\sB_-(K_{X_1}+B_1+M_1)$, and assume, towards a  contradiction, that each map $\pi_i$ is an isomorphism at the generic point of the strict transform of $S$ on $X_i$. Since $S$ is a log canonical centre of $ (X_1,B_1+M_1) $, there exists a divisorial valuation~$T$ over $X_1$ such that $c_{X_1}(T)=S$ and $a(T,X_1,B_1+M_1)={-}1$. By Lemma~\ref{lem:dltblowup} there exists a $\Q$-factorial dlt model $f_1\colon (X_1',B_1' + M_1')\to (X_1,B_1+M_1)$ such that $T$ is a divisor on $X_1'$ which is a component of $\lfloor B_1'\rfloor$. Then there exists a lifted MMP diagram
      	\begin{center}
		\begin{tikzcd}[column sep = 0.8em, row sep = large]
			\left(X_1',B_1'+M_1'\right) \arrow[d, "f_1" swap] \arrow[rr, dashed, "\rho_1"] && \left(X_2',B_2'+M_2'\right) \arrow[d, "f_2" swap] \arrow[rr, dashed, "\rho_2"] && \left(X_3',B_3'+M_3'\right) \arrow[d, "f_3" swap] \arrow[rr, dashed, "\rho_3"] && \cdots 
			\\ 
			(X_1,B_1+M_1) \arrow[rr, dashed, "\pi_1"] && (X_2,B_2+M_2) \arrow[rr, dashed, "\pi_2"] && (X_3,B_3+M_3) \arrow[rr, dashed, "\pi_3"] && \cdots\rlap{.}
		\end{tikzcd}
	\end{center}
   
   Since $S$ is a log canonical centre of the dlt g-pair $ (X_1,B_1+M_1) $, we have that $S$ does not belong to the singular locus $Z_1$ of $X_1$ by Definition~\ref{dfn:singularities}; hence $T\nsubseteq f_1^{-1}(Z_1)$. Since $f_1$ is a dlt blowup, by Lemma~\ref{lem:pullbackdiminished} we have
   $$f^{-1}_1\sB_-(K_{X_1}+B_1+M_1)\cup f_1^{-1}(Z_1)=\sB_-(K_{X_1'}+B_1'+M_1')\cup f_1^{-1}(Z_1),$$
   hence $T\subseteq \sB_-(K_{X_1'}+B_1'+M_1')$. Therefore, $T\subseteq\Supp N_\sigma(K_{X_1'}+B_1'+M_1')$ by Lemma~\ref{lem:nakayamazariskidiminished}. Then, since we assume Conjecture~\ref{con:mainconjecture2}, there exists an index $i$ such that $\rho_i$ is not an isomorphism at the generic point of the strict transform of $T$ on $X_i'$, and we conclude as at the end of Step~\ref{step2} of the proof of Lemma~\ref{lem:liftingtheconjecture}.
\end{proof}

As a corollary we show another fact mentioned in the introduction, that Conjecture~\ref{con:mainconjecture} is a consequence of the termination of flips conjecture.

\begin{cor}\label{cor:B-termination}
Let $ (X_1, B_1+M_1) $ be a projective $\Q$-factorial NQC dlt g-pair such that $K_{X_1}+B_1+M_1$ is pseudoeffective. Consider a $(K_{X_1}+{B_1}+{M_1})$-MMP 
\begin{center}
\begin{tikzcd}[column sep = 0.8em, row sep = 1.75em]
(X_1,B_1+M_1) \arrow[rr, dashed, "\pi_1"] && (X_2,B_2+M_2) \arrow[rr, dashed, "\pi_2"] && (X_3,B_3+M_3) \arrow[rr, dashed, "\pi_3"] && \cdots 
\end{tikzcd}
\end{center}
which terminates. Let $S$ be a log canonical centre of\, $(X_1,B_1+M_1)$ which belongs to $\sB_-(K_{X_1}+B_1+M_1)$. Then there exists an index $i$ such that $\pi_i$ is not an isomorphism at the generic point of the strict transform of $S$ on $X_i$.
\end{cor}

\begin{proof}
We use the same notation as in the second and  third paragraphs of the proof of Proposition~\ref{pro:equivalence}. Assume that each map $\pi_i$ is an isomorphism at the generic point of the strict transform of $S$ on $X_i$, and consider the component $T\subseteq\Supp N_\sigma(K_{X_1'}+B_1'+M_1')$ constructed in that proof. Then $a(E,X_i,B_i+M_i)={-}1$ for each $i$ by  \cite[Lemma 2.8(iii)]{LMT23}; hence $a(E,X_i',B_i'+M_i')={-}1$ for each $i$ as all maps $f_i$ are dlt blowups. But then again by  \cite[Lemma 2.8(iii)]{LMT23} applied to the $ (K_{X_1'} + B_1'+M_1') $-MMP in the proof of Proposition~\ref{pro:equivalence}, we conclude that $T$ is not contracted in that MMP, which contradicts Lemma~\ref{lem:resNQC}\eqref{lem:resNQC-b}.
\end{proof}

\section{Balanced MMP}

In this section we prove Proposition~\ref{pro:balanced}. We first need the following lemma on the behaviour of log canonical thresholds under maps in an MMP; \textit{cf.} \cite[Lemma 1.21]{HM20}.

\begin{lem}\label{lem:ascending}
Let $(X,B+M)$ be a projective $\Q$-factorial NQC log canonical g-pair with a flip or a divisorial contraction
\begin{center}
\begin{tikzcd}[column sep = 0.8em, row sep = 1.75em]
  \pi\colon (X,B+M)\arrow[rr, dashed] &&(X',B'+M'),
\end{tikzcd}
\end{center}
and set $P:=P_\sigma(K_X+B+M)$, $N:=N_\sigma(K_X+B+M)$, $P':=P_\sigma(K_{X'}+B'+M')$, and $N':=N_\sigma(K_{X'}+B'+M')$. If\, $(X,B+M)$ has a minimal model, then for each positive real number $t$, if the g-pair $(X,(B+tN)+(M+tP))$ is log canonical, then the g-pair $(X',(B'+tN')+(M'+tP'))$ is log canonical.
\end{lem}

\begin{proof}
Let $(p,q)\colon W\to X\times X'$ be a smooth resolution of indeterminacies of the map $\pi$. Then by the negativity lemma there exists an effective $q$-exceptional $\R$-divisor $G$ such that
$$p^*(K_X+B+M)\sim_\R q^*(K_{X'}+B'+M')+G,$$
so, by \cite[Lemma 2.16]{GL13} or by \cite[Lemma 2.4]{LP20a}, we have
\begin{equation}\label{eq:89}
N_\sigma\left(p^*(K_X+B+M)\right)= N_\sigma\left(q^*(K_{X'}+B'+M')\right)+G
\end{equation}
and
$$P_\sigma\left(p^*(K_X+B+M)\right)\sim_\R P_\sigma\left(q^*(K_{X'}+B'+M')\right).$$
By passing to a higher resolution, by Lemma~\ref{lem:resNQC}\eqref{lem:resNQC-a} and by \cite[Corollary III.5.17]{Nak04}, we may assume that $P_\sigma(p^*(K_X+B+M))$ is NQC. By Remark~\ref{rem:easycomputation2}, for each real number $t$ and any prime divisor $E$ on $W$, we have
$$a\left(E, X, (B+tN)+(M+tP)\right) = a(E, X, B+M) - t\mult_E N_\sigma\left(p^*(K_X+B+M)\right)$$
and
$$a\left(E, X', (B'+tN')+(M'+tP')\right) = a(E, X', B'+M') - t\mult_E N_\sigma\left(q^*(K_{X'}+B'+M')\right).$$
Since we have $a(E, X, B+M)\leq a(E, X', B'+M')$ by the negativity lemma and since Equation~\eqref{eq:89} implies $N_\sigma(p^*(K_X+B+M))\geq N_\sigma(q^*(K_{X'}+B'+M'))$, we conclude that
$$a\left(E, X, (B+tN)+(M+tP)\right)\leq a\left(E, X', (B'+tN')+(M'+tP')\right),$$
as desired.
\end{proof}

In the following result we examine the behaviour of log canonical thresholds under dlt blowups.

\begin{lem}\label{lem:balanceddltblowup}
Let $(X,B+M)$ be a projective $\Q$-factorial NQC log canonical g-pair with a dlt blowup
$$\pi\colon (X',B'+M') \longrightarrow (X,B+M),$$
and set $P:=P_\sigma(K_X+B+M)$, $N:=N_\sigma(K_X+B+M)$, $P':=P_\sigma(K_{X'}+B'+M')$, and $N':=N_\sigma(K_{X'}+B'+M')$. If\, $(X,B+M)$ has a minimal model, then, for each real number $t$, the g-pair $(X,(B+tN)+(M+tP))$ is log canonical if and only if the g-pair $(X',(B'+tN')+(M'+tP'))$ is log canonical.
\end{lem}

\begin{proof}
Let $f\colon Y\to X'$ be a desingularisation, and set $g:=\pi\circ f$. By Lemma~\ref{lem:resNQC}\eqref{lem:resNQC-a} and by \cite[Corollary~III.5.17]{Nak04}, we may assume that $P_\sigma(g^*(K_X+B+M))=P_\sigma(f^*(K_{X'}+B'+M'))$ is NQC. By applying Remark~\ref{rem:easycomputation2} twice and using $K_{X'}+B'+M'\sim_\R\pi^*(K_X+B+M)$, we see that, for each real number $t$ and any prime divisor $E$ on $Y$,  we have
\begin{align*}
a\left(E, X, (B+tN)+(M+tP)\right)& = a(E, X, B+M) - t\mult_E N_\sigma\left(g^*(K_X+B+M)\right)\\
& = a(E, X', B'+M') - t\mult_E N_\sigma\left(f^*(K_{X'}+B'+M')\right)\\
& = a\left(E, X', (B'+tN')+(M'+tP')\right),
\end{align*}
which proves the lemma.
\end{proof}

The following is an immediate corollary. 

\begin{cor}\label{cor:balancedlifting}
     Let $ (X_1,B_1+M_1) $ be a projective NQC log canonical g-pair such that $K_{X_1}+B_1+M_1$ is pseudoeffective and such that $ (X_1,B_1+M_1) $ has a minimal model. Consider a lifted MMP diagram 
   	\begin{center}
		\begin{tikzcd}[column sep = 0.8em, row sep = large]
			(X_1',B_1'+M_1') \arrow[d, "f_1" swap] \arrow[rr, dashed, "\rho_1"] && (X_2',B_2'+M_2') \arrow[d, "f_2" swap] \arrow[rr, dashed, "\rho_2"] && (X_3',B_3'+M_3') \arrow[d, "f_3" swap] \arrow[rr, dashed, "\rho_3"] && \cdots 
			\\ 
			(X_1,B_1+M_1) \arrow[rr, dashed, "\pi_1"] && (X_2,B_2+M_2) \arrow[rr, dashed, "\pi_2"] && (X_3,B_3+M_3) \arrow[rr, dashed, "\pi_3"] && \cdots\rlap{,}
		\end{tikzcd}
	\end{center}
	where the sequence at the bottom of the diagram is a sequence of flips. Then the MMP at the bottom of the diagram is balanced if and only if the MMP at the top of the diagram is balanced.
\end{cor}

Now we can prove Proposition~\ref{pro:balanced}.

\begin{proof}[Proof of Proposition~\ref{pro:balanced}]
  By Lemmas~\ref{lem:lifting_g} and~\ref{lem:resNQC}\eqref{lem:resNQC-e}, we may assume that all $X_i$ are $\Q$-factorial. For each $i$, let $t_i$ be the log canonical threshold of $P_\sigma(K_{X_i}+B_i+M_i)+N_\sigma(K_{X_i}+B_i+M_i)$ with respect to $(X_i,B_i+M_i)$. Then by Lemma~\ref{lem:ascending} and by \cite[Theorem 1.5]{BZ16}, there exists a positive integer $i_0$ such that $t_i=t_{i_0}$ for all $i\geq i_0$. For each $i\geq i_0$ set $B_i^\circ:=B_i+t_{i_0}N_i$ and $M_i^\circ:=M_i+t_{i_0}P_i$. Then each g-pair $(X_i,B_i^\circ+M_i^\circ)$ is log canonical, the map
\begin{center}
\begin{tikzcd}[column sep = 0.8em, row sep = 1.75em]
  \pi_i\colon\left(X_i,B_i^\circ+M_i^\circ\right)\arrow[rr, dashed] &&\left(X_{i+1},B_{i+1}^\circ+M_{i+1}^\circ\right)
\end{tikzcd}
\end{center}
is a divisorial contraction or a flip, and the sequence of maps $\{\pi_i\}_{i\geq i_0}$ is a balanced MMP for the g-pair $(X_{i_0},B_{i_0}^\circ+M_{i_0}^\circ)$, where this g-pair has a minimal model by Lemma~\ref{lem:resNQC}\eqref{lem:resNQC-d}. Consider a lifted MMP diagram 
   	\begin{center}
		\begin{tikzcd}[column sep = 0.8em, row sep = large]
			\left(X_{i_0}',B_{i_0}'+M_{i_0}'\right) \arrow[d, "f_{i_0}" swap] \arrow[rr, dashed, "\rho_{i_0}"] && \left(X_{i_0+1}',B_{i_0+1}'+M_{i_0+1}'\right) \arrow[d, "f_{i_0+1}" swap] \arrow[rrr, dashed, "\rho_{i_0+1}"] &&& \cdots 
			\\ 
			\left(X_{i_0},B_{i_0}^\circ+M_{i_0}^\circ\right) \arrow[rr, dashed, "\pi_{i_0}"] && \left(X_{i_0+1},B_{i_0+1}^\circ+M_{i_0+1}^\circ\right) \arrow[rrr, dashed, "\pi_{i_0+1}"] &&& \cdots\rlap{.}
		\end{tikzcd}
	\end{center}
Then by Corollary~\ref{cor:balancedlifting} the MMP at the top of this diagram is a balanced MMP of a dlt g-pair $(X_{i_0}',B_{i_0}'+M_{i_0}')$, which has a minimal model by Lemma~\ref{lem:resNQC}\eqref{lem:resNQC-e}. The MMP at the top of this diagram terminates if and only if the original $(K_{X_1}+B_1+M_1)$-MMP terminates, which concludes the proof.
\end{proof}

\section{Special termination}\label{sec:specterm}

As promised in the introduction, in this section we prove a version of special termination in the context of pseudoeffective generalised pairs. 

\begin{thm}\label{thm:specterm_g-pairs1}
Assume the termination of flips for pseudoeffective NQC $\Q$-factorial dlt g-pairs of dimension at most $n-1$.
	
Let $ (X_1,B_1+M_1) $ be a $\Q$-factorial NQC dlt g-pair of dimension $ n $ such that $K_{X_1}+B_1+M_1$ is pseudoeffective. Consider a sequence of flips 
	\begin{center}
		\begin{tikzcd}[column sep = 0.8em, row sep = 1.75em]
			(X_1,B_1+M_1) \arrow[dr, "\theta_1" swap] \arrow[rr, dashed, "\pi_1"] && (X_2,B_2+M_2) \arrow[dl, "\theta_1^+"] \arrow[dr, "\theta_2" swap] \arrow[rr, dashed, "\pi_2"] && (X_3,B_3+M_3) \arrow[dl, "\theta_2^+"] \arrow[rr, dashed, "\pi_3"] && \cdots.  \\
			& Z_1 && Z_2
		\end{tikzcd}
	\end{center}
Assume that Conjecture~\ref{con:mainconjecture} holds for this sequence of flips. Then there exists a positive integer $N$ such that 
	\[ \Exc(\theta_i)\cap\nklt(X_i, B_i + M_i )=\emptyset \quad\text{for all } i\geq N . \]
\end{thm}

\begin{proof}
	We divide the proof into several steps.
	In Steps~\ref{th51-step1} and~\ref{th51-step2} below, we follow closely the proof of \cite[Theorem 5.1]{LMT23}.
	
	We will relabel the indices several times in the proof: this is justified by Lemma~\ref{lem:shiftingtheconjecture}.
	
\begin{enumerate}[wide, label={{\it Step}~\rm\arabic*.}, ref=\arabic*]	\medskip
\item\label{th51-step1}	
From now until the end of the proof, we will be proving the following claim by induction on $d$.
	
\begin{claim} For each non-negative integer $d$, there exists a positive integer $N_d$ such that for each $i\geq N_d$ the restriction of\, $\theta_i$ to each log canonical centre of dimension at most $d$ of the g-pair $(X_i,B_i+M_i)$ is an isomorphism.
\end{claim}

	We first show that the claim implies the theorem. Indeed, the claim implies that $\lfloor B_i\rfloor$ does not \emph{contain} any flipping or flipped curves for all $i\geq N_{n-1}$. Therefore, if $\Exc(\theta_i)\cap\lfloor B_i\rfloor\neq\emptyset$ for some $i\geq N_{n-1}$, then there exists a flipping curve $C\subseteq\Exc(\theta_i)$ such that $C\cdot\lfloor B_i\rfloor > 0$. But then $C^+\cdot\lfloor B_{i+1}\rfloor < 0$ for every flipped curve $C^+\subseteq\Exc(\theta_i^+)$; hence $C^+\subseteq\lfloor B_{i+1}\rfloor$, and we have a contradiction.
	
	\medskip
	
	Now we start proving the claim. First recall  that the number of log canonical centres of any log canonical g-pair is finite. At step $i$ of the MMP as above, if a log canonical centre of $(X_i,B_i+M_i)$ belongs to $\Exc(\theta_i)$, then the number of log canonical centres of $(X_{i+1},B_{i+1}+M_{i+1})$ is smaller than the number of log canonical centres of $(X_i,B_i+M_i)$ by \cite[Lemma 2.8(iii)]{LMT23}. 
	
	Thus, there exists a positive integer $N_0$ such that the set $\Exc(\theta_i)$ does not contain any log canonical centre of $(X_i,B_i+M_i)$ for $i\geq N_0$. By relabelling, we may assume that $N_0=1$. In particular, this proves the claim for $d=0$.
		
	Therefore, we may assume that for each $i\geq1$ the map $\pi_i$ is an isomorphism at the generic point of each log canonical centre of $(X_i,B_i+M_i)$.

\item\label{th51-step2}	 Let $d$ be a positive integer. By induction and by relabelling, we may assume that each map $ \pi_i $ is an isomorphism along every log canonical centre of dimension at most $ d -1 $. 
	
	Now consider a log canonical centre $ S_1 $ of $ (X_1,B_1+M_1) $ of dimension $ d $. By Step~\ref{th51-step1}, we have birational maps $\pi_i|_{S_i}\colon S_i\dashrightarrow S_{i+1}$, where $ S_i $ is the strict transform of $ S_1 $ on $ X_i $. Every log canonical centre of $ (S_i,B_{S_i} + M_{S_i}) $ is a log canonical centre of $(X_i,B_i + M_i)$, and hence by induction each map $\pi_i|_{S_i}$ is an isomorphism along $\Supp \lfloor B_{S_i}\rfloor$. Then by \cite[Proposition 2.17(iv)]{LMT23} and since the difficulty function considered in that result takes values in $\N$, after relabelling the indices, we may assume that
 	\begin{equation}\label{eq:20}
 	\pi_i|_{S_i}\colon S_i\dashrightarrow S_{i+1} \text{ is an isomorphism in codimension $1$ for every $ i $}. 
	\end{equation}
	Moreover, by Lemma~\ref{lem:valuations} and by relabelling the indices,  we may assume that
	\begin{equation}\label{eq:0}
	\left(\pi_i|_{S_i}\right)_*\left(K_{S_i}+B_{S_i}+M_{S_i}\right)=K_{S_{i+1}}+B_{S_{i+1}}+M_{S_{i+1}}\quad\text{for every } i. 	
	\end{equation}

\item\label{th51-step3} Since for each $i$ the map $ \pi_i $ is an isomorphism at the generic point of $S_i$ by Step~\ref{th51-step1}, by Conjecture~\ref{con:mainconjecture} we have that $S_1\nsubseteq\sB_-(K_{X_1}+B_1+M_1)$. Moreover, $S_1$ is normal as it is a log canonical centre of the dlt g-pair $(X_1,B_1+M_1)$. Therefore, Lemma~\ref{lem:restriction2} implies that
\begin{equation}\label{eq:claim2}
\text{the divisor }K_{S_1} + B_{S_1} + M_{S_1}\text{ is pseudoeffective.}
\end{equation}	
For every $ i $ denote by $ R_i $ the normalization of $ \theta_i(S_i) $. By \eqref{eq:0} each diagram
		\begin{center}
		\begin{tikzcd}[column sep = 0.8em, row sep = large]
			\left(S_i,B_{S_i}+M_{S_i}\right) \arrow[dr, "\theta_i|_{S_i}" swap] \arrow[rr, dashed, "\pi_i|_{S_i}"] && \left(S_{i+1},B_{S_{i+1}}+M_{S_{i+1}}\right) \arrow[dl, "\theta_i^+|_{S_{i+1}}"]  \\
			& R_i &
		\end{tikzcd}
	\end{center}
	is a small ample quasi-flip. Thus, by Lemma~\ref{lem:lifting_g} there exists a diagram
		\begin{center}
		\begin{tikzcd}[column sep = 0.8em, row sep = large]
			\left(S_1', B_{S_1'} + M_{S_1'}\right) \arrow[d, "g_1" swap] \arrow[rr, dashed, "\varphi_1"] && \left(S_2', B_{S_2'} + M_{S_2'} \right) \arrow[d, "g_2" swap] \arrow[rrrr, dashed, "\varphi_2"] &&&& \cdots
			\\ 
			\left(S_1,B_{S_1}+M_{S_1}\right) \arrow[dr, "\theta_1|_{S_1}" swap] \arrow[rr, dashed, "\pi_1|_{S_1}"] && \left(S_2,B_{S_2}+M_{S_2}\right) \arrow[dl, "\theta_1^+|_{S_2}"]  \arrow[rrrr, dashed, "\pi_2|_{S_2}"] &&&& \cdots\rlap{,} \\
			& R_1 && 
		\end{tikzcd}
	\end{center}
	where each map $g_i$ is a dlt blowup of $(S_i,B_{S_i}+M_{S_i})$ and the sequence of rational maps $\varphi_i$ yields an MMP for the NQC $ \Q $-factorial dlt g-pair $ (S_1', B_{S_1'} + M_{S_1'}) $. Then $K_{S_1'}+ B_{S_1'} + M_{S_1'}$ is pseudoeffective by \eqref{eq:claim2}.
    Since we assume termination of flips for pseudoeffective NQC $\Q$-factorial dlt g-pairs in dimensions at most $n-1$, we infer that the $(K_{S_1'}+ B_{S_1'} + M_{S_1'})$-MMP terminates. Thus, by relabelling, we may assume that 
    \begin{equation}\label{eq:9}
    \left(S_i', B_{S_i'} + M_{S_i'}\right) \simeq \left(S_{i+1}', B_{S_{i+1}'} + M_{S_{i+1}'}\right) \quad \text{for all } i\geq1 .
    \end{equation}
    The construction yields that $-(K_{S_i'}+ B_{S_i'} + M_{S_i'})$ and $K_{S_{i+1}'}+B_{S_{i+1}'}+M_{S_{i+1}'}$ are nef over $R_i$; hence $K_{S_i'}+ B_{S_i'} + M_{S_i'}$ is numerically trivial over $R_i$ for each $i$ by \eqref{eq:9}. In particular, $K_{S_i}+B_{S_i}+M_{S_i}$ and $K_{S_{i+1}}+B_{S_{i+1}}+M_{S_{i+1}}$ are numerically trivial over $R_i$ for each $i$, and thus $\theta_i|_{S_i}$ and $\theta_i^+|_{S_{i+1}}$ contract no curves. Therefore, $\theta_i|_{S_i}$ and $\theta_i^+|_{S_{i+1}}$ are isomorphisms, and consequently all maps $\pi_i|_{S_i}$ are isomorphisms. This finishes the proof of the claim, and thus of the theorem.
    \hfill\qedhere
    \end{enumerate}
\end{proof}

\section{Proof of the main result}\label{sec:mainresult}

In this section we prove the main result of the paper.

	\begin{proof}[Proof of Theorem~\ref{mainthm}]
		Arguing by contradiction, we assume that the given sequence of flips does not terminate. For each $i\geq1$ we denote by $P_i$, respectively $N_i$, the strict transforms of $P_1:=P_\sigma(K_{X_1}+B_1+M_1)$, respectively $N_1:=N_\sigma(K_{X_1}+B_1+M_1)$, on $X_i$. 
		
		By Lemmas~\ref{lem:resNQC}\eqref{lem:resNQC-d} and~\ref{lem:shiftingtheconjecture}, we may relabel the indices in this sequence; thus by Theorem~\ref{thm:specterm_g-pairs1} we may assume that
		\begin{equation} \label{eq:8}
			\Exc(\theta_i) \cap \nklt(X_i, B_i + M_i ) = \emptyset \quad \text{for all } i .
		\end{equation}
By Lemma~\ref{lem:resNQC}\eqref{lem:resNQC-a} there exists a log resolution $f\colon \widetilde X\to X_1$ of $(X_1,B_1+M_1)$ such that the $\R$-divisor $P_\sigma(f^*(K_{X_1}+B_1+M_1))$ is NQC, and there exists an NQC divisor $\widetilde M$ on $\widetilde X$ such that $M_1 = f_*\widetilde M$. Since the sequence of flips balanced, the log canonical threshold of $ P_1 + N_1 $ with respect to $ (X_1,B_1+M_1) $ is zero; hence by Remark~\ref{rem:easycomputation2} there exists a prime divisor
\begin{equation}\label{eq:10}
E\subseteq\Supp N_\sigma\left(f^*\left(K_{X_1}+B_1+M_1\right)\right)
\end{equation}
such that $a(E, X_1, B_1+M_1)={-}1$. 
	
	Let $\widetilde B$ be the sum of $f_*^{-1}B_1$ and of all $f$-exceptional prime divisors on $\widetilde X$. Then there exists an effective $f$-exceptional $\R$-divisor $G$ on $\widetilde X$ such that
	$$ K_{\widetilde X}+\widetilde B+\widetilde M\sim_\R f^*(K_{X_1}+B_1+M_1)+G,$$
	so, by \cite[Lemma 2.16]{GL13} or by \cite[Lemma 2.4]{LP20a}, we have
	$$N_\sigma\left(K_{\widetilde X}+\widetilde B+\widetilde M\right)= N_\sigma\left(f^*\left(K_{X_1}+B_1+M_1\right)\right)+G.$$
	This equation, together with \eqref{eq:10}, implies that
	\begin{equation}\label{eq:11}
	E\subseteq\Supp N_\sigma\left(K_{\widetilde X}+\widetilde B+\widetilde M\right).
	\end{equation}
	Note that, since $a(E, X_1, B_1+M_1)={-}1$, we have
	\begin{equation}\label{eq:12}
	E\subseteq\Supp\left\lfloor\widetilde B\right\rfloor
	\end{equation}
	and
	\begin{equation}\label{eq:14}
	E\nsubseteq\Supp G.
	\end{equation}
	
	As in the proof of Lemma~\ref{lem:dltblowup}, we can run a $(K_{\widetilde X}+\widetilde B+\widetilde M)$-MMP over $X_1$ which terminates with a model $(X_1', B_1' + M_1')$ and which contracts precisely the divisor $G$. Thus, we obtain a dlt blowup
		$$ h_1 \colon \left(X_1', B_1' + M_1'\right) \longrightarrow \left(X_1, B_1+M_1\right) $$
		of $ (X_1, B_1+M_1) $. The prime divisor $E$ is not contracted by this MMP by \eqref{eq:14}. Let $E_1$ be the strict transform of $E$ on $X_1'$. Then
		\begin{equation}\label{eq:13}
		  E_1\subseteq\Supp N_\sigma
                  \left(K_{X_1'}+B_1'+M_1'\right)\subseteq\sB_-\left(K_{X_1'}+B_1'+M_1'\right)
		\end{equation}
		by \eqref{eq:11} and by Lemmas~\ref{lem:resNQC}\eqref{lem:resNQC-b} and~\ref{lem:nakayamazariskidiminished}, and
		\begin{equation}\label{eq:15}
		E_1\subseteq\Supp\lfloor B_1'\rfloor
		\end{equation}
		by \eqref{eq:12}.
		
	By Lemma~\ref{lem:lifting_g} there exists a diagram
		\begin{center}
		\begin{tikzcd}[column sep = 0.8em, row sep = large]
			\left(X_1', B_1' + M_1'\right) \arrow[d, "h_1" swap] \arrow[rr, dashed, "\rho_1"] && \left(X_2', B_2' + M_2' \right) \arrow[d, "h_2" swap] \arrow[rr, dashed, "\rho_2"] && \cdots
			\\ 
			(X_1, B_1+M_1) \arrow[rr, dashed, "\pi_1"] && (X_2, B_2+M_2) \arrow[rr, dashed, "\pi_2"] && \cdots\rlap{,}
		\end{tikzcd}
	\end{center}
	where for each $ i\geq 1$ the map $\rho_i\colon X_i'\dashrightarrow X_{i+1}'$ is a $ (K_{X_i}' + B_i' + M_i') $-MMP and the map $h_i$ is a dlt blowup of the g-pair $(X_i,B_i + M_i)$. In particular, the sequence at the top of the above diagram is an MMP for the NQC $\Q$-factorial dlt g-pair $(X_1',B_1'+M_1')$, and by Lemma~\ref{lem:liftingtheconjecture} we have that Conjecture~\ref{con:mainconjecture} holds for that sequence. Therefore, \eqref{eq:13} and \eqref{eq:15} imply that $E_1$ is contracted at some step of that sequence. However, by \eqref{eq:8} and by \cite[Lemma 2.20(i)]{CT23}, that MMP is a sequence of flips, and we have a contradiction.
	\end{proof}
 

\end{document}